\documentclass[
final
]{dmtcs-episciences}

\usepackage[utf8]{inputenc}
\usepackage{subfigure}
\usepackage{enumerate}

\usepackage{amssymb}
\usepackage{amsthm}
\usepackage{amsmath}
\theoremstyle{plain}
\newtheorem{theorem}{Theorem}[section]

\newtheorem{conjecture}[theorem]{Conjecture}
\newtheorem{proposition}[theorem]{Proposition}
\newtheorem{lemma}[theorem]{Lemma}
\newtheorem{corollary}[theorem]{Corollary}
\newtheorem{claim}{Claim}
\newtheorem{problem}[theorem]{Problem}

\usepackage[round]{natbib}

\author{Carl Johan Casselgren\affiliationmark{1}
\thanks{Casselgren was supported by a grant
from the Swedish Research Council (2017-05077).}
    \and  Petros A. Petrosyan\affiliationmark{2}}
\title[Some results on the palette index of graphs]
{Some results on the palette index of graphs\footnote
{A preliminary version of some of the results in this paper appeared
in the proceedings of the conference CSIT 2017, Yerevan, Armenia}}
\affiliation{
  Link\"oping University, Sweden\\
  Yerevan State University, Armenia}
\received{2018-5-16}

\revised{2019-3-1}

\accepted{2019-3-4}

\begin{document}
\publicationdetails{21}{2019}{3}{11}{4509}
\maketitle
\begin{abstract}
 	Given a proper edge coloring $\varphi$ of a graph $G$, we define
	the palette $S_{G}(v,\varphi)$ of a vertex $v \in V(G)$ as
	the set of all colors appearing on edges incident with $v$.
	The palette index $\check s(G)$ of $G$ is the minimum number
	of distinct palettes occurring in a proper edge coloring of $G$.
	In this paper we give various upper and lower bounds on the
	palette index of $G$ in terms of the vertex degrees of $G$,
	particularly for the case when $G$ is a bipartite graph
	with small vertex degrees.
	Some of our results concern $(a,b)$-biregular graphs; that is,
	bipartite graphs where all vertices in one part have degree $a$
	and all vertices in the other part have degree $b$.
	We conjecture that if $G$ is $(a,b)$-biregular, then
	$\check{s}(G)\leq 1+\max\{a,b\}$, and we prove that this
	conjecture holds for several families of $(a,b)$-biregular graphs.
	Additionally, we characterize the graphs whose palette index
	equals the number of vertices.
\end{abstract}

\section{Introduction}
Given an edge coloring $\varphi$ of a graph $G$, we define
the {\it palette} $S_{G}(v,\varphi)$
(or just $S(v,\varphi)$)
 of a vertex $v \in V(G)$ as
the set of all colors appearing on edges incident with $v$.
The {\em palette index} $\check s(G)$ of $G$ is the minimum number
of distinct palettes occurring in a proper edge coloring of $G$.
This notion was introduced quite recently 
by \cite{HornakKalinowskiMeszkaWozniak} 
and has so far primarily been studied for the
case of regular graphs.

Denote by $\Delta(G)$ and $\chi'(G)$ the maximum degree
and the chromatic index of a graph $G$, respectively.
By Vizing's well-known edge coloring theorem $\chi'(G) = \Delta(G)$
or $\chi'(G)=\Delta(G)+1$ for every graph $G$. In the former case
$G$ is said to be {\em Class 1}, and in the latter case $G$ is
{\em Class 2}.

Trivially, $\check s(G) = 1$ if and only $G$ is a regular
Class $1$ graph, and by Vizing's edge coloring theorem
it holds that if $G$ is regular and Class 2, then
$3 \leq \check s(G) \leq \Delta(G)+1$; the case
$\check s(G) = 2$ is not possible, 
as proved in \cite{HornakKalinowskiMeszkaWozniak}.

Since computing the chromatic index of a given graph
is $\mathcal{NP}$-complete, as proved in \cite{LevenGalil},
determining the palette index of a given
graph is $\mathcal{NP}$-complete, even for $3$-regular graphs.
Note further that this in fact means that even determining if a given graph
has palette index $1$ is an $\mathcal{NP}$-complete problem.
Nevertheless, in \cite{HornakKalinowskiMeszkaWozniak} it was proved that
the palette index of a cubic Class 2 graph is $3$ or $4$ according
to whether the graph has a perfect matching or not.
 \cite{BonviciniMazzuoccolo} 
investigated 4-regular graphs; they
proved that $\check s(G) \in \{3,4,5\}$
if $G$ is $4$-regular and Class 2, and that
all these values are in fact attained.

Vizing's edge coloring theorem yields an upper bound on the palette index 
of a general graph $G$ with maximum
degree $\Delta$ and no isolated vertices, 
namely that $\check s(G) \leq 2^{\Delta+1}-2$.
However, this is probably far from being tight. Indeed,
\cite{AvesaniBonisoliMazzuoccolo} 
described an infinite family
of multigraphs whose palette index grows 
asymptotically as $\Delta^2$; it is
an open question whether there are 
such examples without multiple edges.
Furthermore, they
suggested to prove that there is a polynomial 
$p(\Delta)$ such that for any graph with maximum 
degree $\Delta$, $\check s(G) \leq p(\Delta)$.
In fact, they suggested that such a polynomial is quadratic in $\Delta$.
We thus arrive at the following conjecture:

\begin{conjecture}
\label{conj:Avesani}
	There is a constant $C$, such that 
	for any graph $G$ with maximum degree $\Delta$,
	$\check s(G) \leq C\Delta^2$.
\end{conjecture}

Very little is known about the palette index of non-regular graphs.
 \cite{BonisoliBonviciniMazzuoccolo} 
studied the palette index of trees, and quite recently
\cite{HornakHudak} completely determined
the palette index of complete bipartite graphs $K_{a,b}$
with $a \leq 5$.

In this note we study the palette index of some families
of non-regular graphs. Before outlining the results of this paper,
let us briefly consider a connection to another kind of edge coloring.

\bigskip

An {\em interval $t$-coloring} of a graph $G$ is a proper $t$-edge coloring such
that for every vertex $v$ of $G$ the colors of the 
edges incident with $v$ form an
interval of consecutive integers; if we also add the condition that
color $1$ is considered as consecutive of color $t$, then we get a
{\em cyclic interval $t$-coloring}. Note that any graph $G$
with an interval coloring admits a cyclic interval $\Delta(G)$-coloring
(by taking all colors modulo $\Delta(G)$).

As noted in \cite{AvesaniBonisoliMazzuoccolo},
if a graph $G$ with maximum degree $\Delta$
has an interval coloring, then
$\check s(G) \leq \Delta^2-\Delta+1$. Moreover, this upper bound holds
for graphs with a cyclic interval $\Delta$-coloring (as implicit in
the proof in \cite{AvesaniBonisoliMazzuoccolo}).
In fact, it holds that for any graph $G$ with maximum degree $\Delta$,
if $G$ has a cyclic interval $C\Delta$-coloring, where $C$ is some absolute constant,
then the palette index of $G$ is bounded by a quadratic polynomial
in $\Delta$. 
An example of a family of graphs with this property
(which do not in general admit interval colorings)
are complete multipartite graphs; 
such a graph $G$ has a cyclic interval coloring
with at most $2\Delta(G)$ colors, as proved in
\cite{AsratianCasselgrenPetrosyanJGT}.
Since there are at most $\Delta$ different vertex degrees
in a graph with maximum degree $\Delta$,
it follows that
Conjecture \ref{conj:Avesani} is true for
every complete multipartite graph.

\begin{proposition}
	If $G$ is a complete multipartite graph with maximum degree
	$\Delta$, then $\check s(G) \leq 2 \Delta^2$.
\end{proposition}

We do not know of any cyclically interval colorable graph $G$
that requires more than $2 \Delta(G)$ colors for a cyclic interval coloring;
thus we suggest that
Conjecture \ref{conj:Avesani} particularly holds for
any graph with a cyclic interval coloring.
Note further that it is in fact an open problem to determine
if there is a graph $G$ that requires more than $\Delta(G) +1$ colors for a cyclic interval coloring
(cf. \cite{CasselgrenKhachatrianPetrosyan}).

In the following, we shall present some further upper bounds on the palette
index based on
connections with cyclic interval colorings;
as it turns out, existence of cyclic interval colorings do in fact
provide tight upper bounds on the palette index of some
families of graphs.
Furthermore, motivated by the connection with cyclic interval colorings,
we consider the problem of determining the palette index
of a natural generalization
of regular bipartite graphs, namely so-called 
{\em $(a,b)$-biregular graphs}, i.e.,
bipartite graphs where all vertices in one part have degree $a$ and
all vertices in the other part have degree $b$. 
Note that regular bipartite graphs trivially have 
cyclic interval colorings;
it has been conjectured in \cite{CasselgrenToft} 
that this also holds for $(a,b)$-biregular graphs.

\begin{conjecture}
\label{conj:cycbiregular}
	Every $(a,b)$-biregular graph admits a cyclic interval 
	$\max\{a,b\}$-coloring.\footnote{See e.g. 
	\cite{AsratianCasselgrenPetrosyanJGT}, for further information
on the status of this conjecture.}
\end{conjecture}

The general problem of determining the palette index of a given
$(a,b)$-biregular graph is $\mathcal{NP}$-complete; this follows e.g.
from the complexity result in \cite{AsratianCasselgren}.
We would like to suggest the following weakening of Conjecture
\ref{conj:cycbiregular}, which is a strengthening of
Conjecture \ref{conj:Avesani} for biregular graphs.

\begin{conjecture}
\label{conj:palettebiregular}
For any $(a,b)$-biregular graph $G$, $\check{s}(G)\leq
1+\max\{a,b\}$.
\end{conjecture}

Note that the upper bound in Conjecture \ref{conj:palettebiregular}
is in general tight, since $\check s(G) = b+1$ if $G$ is $(1,b)$-biregular.
However, as we shall see, the upper bound in Conjecture 
\ref{conj:palettebiregular} can be slightly improved for some
values of $a$ and $b$.

Let us now outline the main results of this paper.
We shall present
several results towards
Conjectures \ref{conj:Avesani} and \ref{conj:palettebiregular}.
In the next section we prove a general upper bound on the palette index
of bipartite graphs and deduce that
Conjecture \ref{conj:Avesani} holds for bipartite graphs where
all vertex degrees are in the set
$\{1,2,3,4,2r-4,2r-3,2r-2, 2r-1, 2r\}$, for some $r \geq 1$.
Additionally, we demonstrate that Conjecture \ref{conj:Avesani}
is true for general graphs $G$ satisfying that $\Delta(G)-\delta(G) \leq 2$,
where $\delta(G)$ denotes the minimum degree of $G$.

In Section 3 we consider bipartite graphs with small vertex degrees.
In particular, we obtain sharp upper bounds on the palette indices of
Eulerian bipartite graphs with maximum degree at most $6$. We
also determine the palette index of grids.

Section 4 concerns biregular graphs and Conjecture 
\ref{conj:palettebiregular}. We prove that this
conjecture holds for all $(2,r)$-biregular
and $(2r-2,2r)$-biregular graphs. Additionally, we
establish that it holds
for all $(a,b)$-biregular graphs such that
\begin{itemize}
	
	\item $(a,b) \in \{(3,6), (3,9)\}$;
	 
	\item $(a,b) \in \{(4,6), (4,8), (4,12), (4,16)\}$;
	
	\item $(a,b) \in \{(5,10),(6,9), (6,12)\}$;
	
	\item $(a,b) \in \{(8,12), (8,16), (12,16)\}$.

\end{itemize}

Finally, as mentioned above, $\check s(G) =1$ if and only if $G$ is regular
and Class 1; in Section 5 we characterize the graphs 
whose palette index is at the opposite end of the spectrum; that is,
we give a complete characterization of the graphs whose palette index equals the number of vertices.\\

\section{General upper bounds}

As noted earlier, Vizing's edge coloring theorem yields an
upper bound of the palette index of a general graph, and
K\"onig's edge coloring theorem shows that
this general upper bound can be slightly improved for bipartite graphs:
$\check s(G) \leq 2^\Delta-1$ for any bipartite graph $G$ with maximum
degree $\Delta$ and no isolated vertices. 
In the following we shall give an improvement of this general upper bound for bipartite graphs. Throughout, we assume that all graphs in this section
do not contain any isolated vertices.

We shall need a classic result from factor theory.
A \emph{$2$-factor} of a multigraph $G$ (where loops are allowed) is a
$2$-regular spanning subgraph of $G$.

\begin{theorem} (Petersen's Theorem).
\label{th:Petersen}
	Let $G$ be a $2r$-regular
	multigraph (where loops are allowed). Then $G$ has a decomposition
	into edge-disjoint $2$-factors.
\end{theorem}

For a graph $G$, denote by $D(G)$ the set of all degrees in $G$,
and by $D^{\text{odd}}(G)$ ($D^{\text{even}}(G)$) the set of all odd (even) degrees in $G$.
A graph is {\em even (odd)} if all vertex degrees of the graph are even (odd).

\begin{theorem}
\label{prop:bipevendeg}
	If $G$ is an even bipartite graph, then
	$$\check s(G) \leq \sum_{d \in D(G)} \binom{\frac{\Delta(G)}{2}}
	{\frac{d}{2}}.$$
\end{theorem}

\begin{proof}
	For the proof, we construct a new multigraph $G^{\star}$ as follows: for each vertex $u \in
V(G)$ of degree $2k$, we add $\frac{\Delta(G)}{2}-k$ loops at $u$ $\left(1\leq k< \frac{\Delta(G)}{2}\right)$. Clearly, $G^{\star}$ is a $\Delta(G)$-regular
multigraph. By Petersen's theorem, $G^{\star}$ can be represented as a union
of edge-disjoint $2$-factors $F_{1},\ldots,F_{\frac{\Delta(G)}{2}}$. By removing all
loops from $2$-factors $F_{1},\ldots,F_{\frac{\Delta(G)}{2}}$ of $G^{\star}$, we
obtain that the resulting graph $G$ is a union of edge-disjoint
even subgraphs $F^{\prime}_{1},\ldots,F^{\prime}_{\frac{\Delta(G)}{2}}$. Since
$G$ is bipartite, for each $i$ $\left(1\leq i\leq \frac{\Delta(G)}{2}\right)$,
$F^{\prime}_{i}$ is a collection of even cycles in $G$, and
we can properly color the edges of $F^{\prime}_{i}$ alternately with colors
$2i-1$ and $2i$; the obtained coloring $\alpha$
is a proper edge coloring of $G$ with colors $1,\ldots,\Delta(G)$. 

Now, if $u\in V(G)$ and $d_{G}(u)=2k$,
then there are $k$
even subgraphs $F^{\prime}_{i_{1}},F^{\prime}_{i_{2}},\ldots,F^{\prime}_{i_{k}}$ such that
$d_{F^{\prime}_{i_{1}}}(u)=d_{F^{\prime}_{i_{2}}}(u)=\cdots=d_{F^{\prime}_{i_{k}}}(u)=2$, and thus 
$S_{G}(u,\alpha)=\{2i_{1}-1,2i_{1},2i_{2}-1,2i_{2},\ldots,2i_{k}-1,2i_{k}\}$.
This implies that for vertices $u\in
V(G)$ with $d_{G}(u)=2k$,
we have at most $\binom{\frac{\Delta(G)}{2}}{k}$ distinct palettes
in the coloring $\alpha$.  
\end{proof}

In the next two sections, we shall see that Theorem \ref{prop:bipevendeg}
can in fact be used to deduce sharp upper bounds on the palette
index of some classes of bipartite graphs.

\bigskip

From a given bipartite graph $G$ we can construct an even supergraph 
$G'$ by taking two vertex-disjoint copies $G_1$ and $G_2$
of $G$ and for every odd-degree vertex of $G_1$ joining it by an 
edge with its copy in $G_2$. By applying the preceding proposition to 
$G'$ we immediately obtain the following.

\begin{corollary}
\label{prop:bipodddegree}
	If $G$ is a bipartite graph,
		then
	$$\check s(G) \leq \sum_{d \in D^{\text{odd}}(G)} 
	\binom{\left\lceil\frac{\Delta(G)}{2}\right\rceil}{\frac{d+1}{2}}
	\times  (d+1) 
	+
	\sum_{d \in D^{\text{even}}(G)} 
	\binom{\left\lceil\frac{\Delta(G)}{2}\right\rceil}{\frac{d}{2}}.$$
\end{corollary}

\begin{proof}
	Consider the graph $G'$ defined above, and a proper edge coloring
	$\alpha$ of $G'$ defined as in the proof of Theorem \ref{prop:bipevendeg}.	
	For each palette $S_{G'}(v, \alpha)$ in $G'$, where $v \in D^{\text{odd}}(G)$,
	there are at most $(d_G(v)+1)$ possible palettes in
	the restriction of $\alpha$ to $G$.
\end{proof}

Using Corollary \ref{prop:bipodddegree}, we deduce an improvement of the
general upper bound $2^{\Delta(G)}-1$
on the palette index of any bipartite graph.

\begin{corollary}
\label{cor:bipevendeg}
	For any bipartite graph $G$, $\check s(G) \leq 
	(\Delta(G)+2)2^{\left\lceil\Delta(G)/2 \right\rceil}$.
\end{corollary}

\bigskip

As noted above, the palette index of a regular Class 1 graph is $1$.
We note that Corollary \ref{prop:bipodddegree} 
implies that Conjecture \ref{conj:Avesani}
holds for
bipartite graphs that are ``almost
regular'' in the sense that if
$G$ is a bipartite graph where
all vertex degrees are in the set
	$\{1,2,3,4,2r-4,2r-3,2r-2, 2r-1, 2r\}$, for some $r \geq 4$,
then $G$ satisfies Conjecture \ref{conj:Avesani}.
For general graphs, a slightly weaker proposition is true.

\begin{proposition}
\label{prop:cycdef}
	If a graph $G$ satisfies that $\Delta(G)-\delta(G) \leq 2$,
	then $\check s(G) \leq \Delta^2(G) + \Delta(G) +1$.
\end{proposition}

The proof of this proposition is along the same lines as the proof of
Theorem 5.9 in \cite{AsratianCasselgrenPetrosyanCycDef};
for the sake of completeness, we provide a brief sketch here.

\begin{proof}[(sketch)]
	If $\Delta(G)-\delta(G) \leq 1$, or $G$ is Class 1, 
	then the proposition clearly holds; indeed if $G$ is Class 1,
	then $\check s(G) \leq \binom{\Delta(G)}{2} + \Delta(G)+1 
	\leq \Delta^2(G) + \Delta(G) +1$.
	
	So assume that $\Delta(G) = \delta(G)+2$, and that $G$ is Class 2.
	Set $k = \Delta(G)$ and denote by $V_i$ the set of vertices in 
	$G$ that have degree $i$.
	
	Let 
	$M$ be a maximum matching of $G[V_{k}]$.
	Set $H = G-M$. Note that in $H$ no two vertices
	of degree $k$ in $H$ are adjacent, so 
	by a well-known result due to  \cite{Fournier},
	$H$ is Class 1.
	Let $M'$ be a minimum matching in $H$
	covering all vertices of degree $k$ in $H$; such a matching
	exists since $H$ is Class 1. Note that the graph  $J=H-M'$ has maximum
	degree at most $k-1$. Let $M''$ be a maximum matching in $J_{k-1}$,
	where $J_{k-1}$ is the subgraph of $J$ induced by the vertices of 
	degree $k-1$ in $J$.  Let $\hat M = M \cup M' \cup M''$.
	The rest of the proof is based on the following two claims,
	the proofs of which are omitted (for details, see 
	\cite{AsratianCasselgrenPetrosyanCycDef}).

	\begin{claim}
	\label{cl:GM}
				The subgraph of $G$ induced by $\hat M$ is $2$-edge-colorable.
	\end{claim}
	
		\begin{claim}
	\label{cl:G-M}
				The graph $G - \hat M$ is $(k-1)$-edge-colorable.
	\end{claim}
	
	Let $\psi$ be a proper $(k-1)$-edge coloring of $G -\hat M$
	using colors $1,\dots k-1$, and let $\varphi$ be a proper 
	$2$-edge coloring of the subgraph of  $G$ 
	induced by $\hat M$ using colors $k$ and $k+1$.
	Denote by $\alpha$ the edge coloring of $G$ obtained 
	by taking the two edge colorings $\psi$ and $\varphi$ together.
	
	Since a vertex of degree $k-2$ in $G$ is incident with at
	most one edge from $\hat M$, there are $2\binom{k-1}{k-3} + (k-1)$ 
	possible palettes under $\alpha$;
	a vertex of degree $k-1$ in $G$ is incident with at most
	one edge from $\hat M$ and thus there are most $2(k-1) +1$ possible palettes
	under $\alpha$;
	a vertex of degree $k$ in $G$ is incident with one or two edges from $\hat M$
	and thus there are at most $2+(k-1)$ possible palettes.
\end{proof}

Finally, let us remark that every graph where all vertex degrees  are in the set
$\{1,2,r-2, r-1, r\}$, for some $r\geq 5$,
also satisfies 
Conjecture \ref{conj:Avesani}.\\


\section{Bipartite graphs with small vertex degrees}

In this section we consider bipartite graphs with small
vertex degrees.
As above, throughout this section we assume that all graphs do not
contain any isolated vertices.
We begin this section by noting some immediate implications of Theorem
\ref{prop:bipevendeg}.

\begin{corollary}
\label{cor:bipeuldeg4}
	If $G$ is an Eulerian bipartite graph with $\Delta(G)=4$, then
	$\check s(G) \leq 3$.
\end{corollary}

If $G$ is bipartite, Eulerian, has maximum degree $4$, and there
is a vertex of degree $4$ in $G$ which 
is adjacent to at least three vertices of degree two,
then $\check s(G) \geq 3$; for instance $\check s(K_{2,4}) \geq 3$, so
the upper bound in Corollary \ref{cor:bipeuldeg4}
is sharp.

\begin{corollary}
\label{cor:bipdeg4}
	If $G$ is a bipartite graph with $\Delta(G) =4$, then $\check s(G) \leq 11$.
	Moreover, if $G$ has no pendant vertices, then $\check s(G) \leq 7$.
\end{corollary}

\begin{proof}
	Starting from two copies of $G$, 
	we can create an Eulerian bipartite graph $G'$
	with maximum degree $4$
	containing $G$ as a subgraph.
	Let $\varphi$ be a proper $4$-edge coloring of $G'$ constructed as
	in the proof of Theorem \ref{prop:bipevendeg}, and
	let us consider the restriction
	of this edge coloring to $G$. Vertices 
	of degree $4$ all have the same palette, vertices
	of degree $2$ in $G$ have at most
	$2$ distinct possible palettes;
	vertices of degree $3$ in $G$ have at most $4$ distinct palettes, 
	and similarly for
	vertices of degree $1$.
\end{proof}

We note that the preceding corollary is sharp, which follows by considering
a disjoint union of $K_{1,4}$, $K_{2,4}$, and $K_{3,4}$:
the palette indices of these graphs are $5$, $3$ and $5$, respectively, 
as observed in
\cite{HornakHudak}; in fact, in any proper edge coloring
of this graph
the vertices of degree $1$ have four distinct palettes, vertices of degree
$2$ have at least two distinct palettes, vertices of degree three have
four different palettes, and vertices of degree four have at least one 
palette.
Hence, the palette index of the disjoint union
of these complete bipartite graphs is at least $11$.

From Corollary \ref{cor:bipdeg4} we deduce an upper bound on 
the palette index of bipartite graphs with maximum degree $5$.

\begin{corollary}
\label{prop:bipdeg5}
	If $G$ is a bipartite graph with $\Delta(G) =5$, then $\check s(G) \leq 23$.
	Moreover,
			if $G$ has a perfect matching, then $\check s(G) \leq 12$.
	\end{corollary}
\begin{proof}
Let $M$ be minimal matching in $G$ covering all vertices of degree $5$; such
a matching exists e.g. by K\"onig's edge coloring theorem. 
By Corollary \ref{cor:bipdeg4}, $G-M$ has a proper edge coloring with
$4$ colors and
at most $11$ distinct palettes; by assigning a new color $5$
to all edges of $M$, we obtain a proper edge coloring of $G$ with
at most $23$ distinct palettes, because for any palette in $G-M$, we obtain
at most $2$ different palettes in $G$, and additionally, the palette $\{5\}$.

The second part follows by applying Corollary \ref{cor:bipdeg4}
to the graph $G-M'$, where $M'$ is
a perfect matching in $G$.
\end{proof}

For Eulerian bipartite graphs with maximum degree six we have the following
immediate consequence of Theorem \ref{prop:bipevendeg}.

\begin{corollary}
	If $G$ is an Eulerian bipartite graph with $\Delta(G) =6$, then
	$\check s(G) \leq 7$.
\end{corollary}

Consider a graph that is the disjoint
union of $K_{2,6}$ and $K_{4,6}$. \cite{HornakHudak} proved that
$\check s(K_{2,6})=4$, and
$\check s(K_{4,6})=4$ ,
which, as above, implies that the upper
bound in the preceding corollary is sharp.

Note further that the preceding corollary shows that
Conjecture \ref{conj:palettebiregular} holds for $(4,6)$-biregular graphs.

For Eulerian bipartite
graphs $G$ with maximum degree $8$, Theorem \ref{prop:bipevendeg}
implies that $\check s(G) \leq 15$.
Using a result from
\cite{AsratianCasselgrenPetrosyanJGT}
we deduce that in fact a better upper bound holds:

\begin{proposition}
\label{prop:bipdeg8}
	If $G$ is an Eulerian bipartite graph with maximum degree $8$,
	then $\check s(G) \leq 13$.
\end{proposition}
The proof is omitted since it immediately follows from the proof of
Theorem 3 in \cite{AsratianCasselgrenPetrosyanJGT}.

\bigskip

Our final result in this section concerns a particular family
of bipartite graphs.
The grids $G(m,n)$ are Cartesian products of paths on $m$ and $n$ vertices, respectively. Here, we determine the exact value of the palette index of $G(m,n)$.

\begin{theorem}
\label{mytheorem14} For any $m,n\geq 2$, 
\begin{center}
$\check s(G(m,n))=\left\{
\begin{tabular}{ll}
$1$, & if $m=n=2$,\\
$2$, & if $\min\{m,n\}=2$ and $\max\{m,n\}\geq 3$,\\
$3$, & if $m,n\geq 3$ and $m n$ is even,\\
$5$, & if $m,n\geq 3$ and $m n$ is odd. \\
\end{tabular}%
\right.$
\end{center}
\end{theorem}
\begin{proof} Let $V(G(m,n)) = \left\{v_j^{(i)}: 1\leq i\leq m,1\leq j \leq n\right\}$ and
$$E(G(m,n))=\left\{v_{j}^{(i)}v_{j+1}^{(i)}: 1\leq i\leq m,1\leq j\leq n-1\right\} \cup 
\left\{v_{j}^{(i)}v_{j}^{(i+1)}: 1\leq i\leq m-1,1\leq j\leq n\right\}.$$

First we show that if $m n$ is even, then
\begin{center}
$\check s(G(m,n))=\left\{
\begin{tabular}{ll}
$1$, & if $m=n=2$,\\
$2$, & if $\min\{m,n\}=2$ and $\max\{m,n\}\geq 3$,\\
$3$, & if $m,n\geq 3$ and $m n$ is even.\\
\end{tabular}%
\right.$
\end{center}

Trivially, $\check s(G(2,2))=\check s(C_{4})=1$. So, without loss of generality we may assume that $\max\{m,n\}\geq 3$ and $m$ is even. Define an edge coloring $\alpha$ of $G(m,n)$ as follows:

\begin{description}
\item[(1)] for $i=1,\ldots,m$, $j=1,\ldots,n-1$, let
\begin{center}
$\alpha\left(v_{j}^{(i)}v_{j+1}^{(i)}\right)=\left\{
\begin{tabular}{ll}
$2$, & if $j$ is odd,\\
$1$, & if $j$  is even;\\
\end{tabular}%
\right.$
\end{center}
\item[(2)] for $i=1,\ldots,\frac{m}{2}$, $j=1,\ldots,n-1$, let
\begin{center}
$\alpha\left(v_{j}^{(2i-1)}v_{j}^{(2i)}\right)=\left\{
\begin{tabular}{ll}
$1$, & if $j=1$,\\
$3$, & otherwise;\\
\end{tabular}%
\right.$
\end{center}
\item[(3)] for $i=1,\ldots,\frac{m}{2}-1$, $j=1,\ldots,n$, let
\begin{center}
$\alpha\left(v_{j}^{(2i)}v_{j}^{(2i+1)}\right)=\left\{
\begin{tabular}{ll}
$3$, & if $j=1$ or $j=n$,\\
$4$, & otherwise;\\
\end{tabular}%
\right.$
\end{center}
\item[(4)] for $i=1,\ldots,\frac{m}{2}$, let
\begin{center}
$\alpha\left(v_{n}^{(2i-1)}v_{n}^{(2i)}\right)=\left\{
\begin{tabular}{ll}
$2$, & if $n$ is odd,\\
$1$, & if $n$  is even,\\
\end{tabular}%
\right.$\\
\end{center}
\end{description}

It is easy to see that $\alpha$ is proper edge coloring of $G(m,n)$ with colors $1,2,3,4$, such that for each vertex $v\in V(G(m,n))$, $S(v,\alpha)\in \{\{1,2\},\{1,2,3\},\{1,2,3,4\}\}$. This shows that if $\max\{m,n\}\geq 3$ and 
$m n$ is even, then
\begin{center}
$\check s(G(m,n))=\left\{
\begin{tabular}{ll}
$2$, & if $\min\{m,n\}=2$ and $\max\{m,n\}\geq 3$,\\
$3$, & if $m,n\geq 3$ and $m n$ is even.\\
\end{tabular}%
\right.$
\end{center} 

\bigskip

Next we consider the case $m,n\geq 3$ and $m n$ is odd.
We first prove the upper bound, i.e. that $\check s(G(m,n))\leq 5$.
Without loss of generality we may assume that $m\leq n$. Let us first show that $\check s(G(3,n))\leq 5$.

 Define an edge coloring $\beta$ of $G(3,n)$ as follows:

\begin{description}
\item[1)] for $i=1,2,3$,  $j=1,\ldots,n-1$, let
\begin{center}
$\beta\left(v_{j}^{(i)}v_{j+1}^{(i)}\right)=\left\{
\begin{tabular}{ll}
$2$, & if $i=1$ and $j$ is odd,\\
$1$, & if $i=1$ and $j$  is even,\\
$2$, & if $i=2$ and $j$ is odd,\\
$4$, & if $i=2$ and $j$  is even,\\
$4$, & if $i=3$ and $j$ is odd,\\
$2$, & if $i=3$ and $j$  is even;\\
\end{tabular}%
\right.$
\end{center}
\item[2)] $j=2,\ldots,n-1$, let
\begin{center}
$\beta\left(v_{j}^{(1)}v_{j}^{(2)}\right)=3$ and $\beta\left(v_{j}^{(2)}v_{j}^{(3)}\right)=1$;
\end{center}
\item[3)] 
$\beta\left(v_{1}^{(1)}v_{1}^{(2)}\right)=\beta\left(v_{n}^{(2)}v_{n}^{(3)}\right)=1$, $\beta\left(v_{n}^{(1)}v_{n}^{(2)}\right)=2$ and $\beta\left(v_{1}^{(2)}v_{1}^{(3)}\right)=3$.
\end{description}
It is not difficult to see that $\beta$ is proper edge coloring of $G(3,n)$ with colors $1,2,3,4$ such that for each vertex $v\in V(G(3,n))$, $S(v,\beta)\in \{\{1,2\},\{3,4\},\{1,2,3\},\{1,2,4\},\{1,2,3,4\}\}$.

If $m\geq 5$, then we 
define a proper edge coloring of $G(m,n)$ in the following way:
let $H = G(m,n) - \left\{v_i^{(m-3)}v_i^{(m-2)} : 1\leq i \leq n\right\}$.
The graph $H$ consists of two components
$H_1$ and $H_2$, where $H_1$ is
isomorphic to $G(m-3,n)$, and $H_2$ is isomorphic to $G(3,n)$.
Let $\alpha'$ be a proper edge coloring of $H_1$ corresponding to the
coloring $\alpha$ of $G(m-3,n)$ defined above, and let $\beta'$ be a proper
edge coloring of $H_2$ corresponding to the edge coloring $\beta$
of $G(3,n)$ defined above. Suppose further that these edge colorings are chosen
in such a way that
vertices $v_{1}^{(m-3)},v_{2}^{(m-3)},\ldots,v_{n}^{(m-3)}$ of $H_1$ 
have the same palettes as vertices $v_{1}^{(m-2)},v_{2}^{(m-2)},\ldots,v_{n}^{(m-2)}$ of $H_2$. 
Thus, by coloring all edges of $G(m,n)$ with one endpoint in $H_1$
and one endpoint in $H_2$ with color $4$, we obtain a proper edge coloring of $G(m,n)$
with $5$ palettes; thus $\check s(G(m,n))\leq 5$.

\bigskip

We now turn to the lower bound. Since  $m,n\geq 3$ and $m n$ is odd,
the graph $G(m,n)$ contains vertices of degree $2,3$ and $4$;
hence $\check s(G(m,n))\geq 3$.

Next, we prove that
 $\check s(G(m,n))\geq 4$. Let $\gamma$ be a proper edge coloring of $G(m,n)$ with three distinct palettes. This implies that for each vertex $v\in V(G(m,n))$ with
degree four, we have $S(v,\gamma)=\{a,b,c,d\}$. 
Let $M_a,M_b,M_c$ and $M_d$ be the color classes of $\gamma$ corresponding 
to the colors $a,b,c$ and $d$. 
Now, there are precisely $(m-2)(n-2)$ vertices of degree four in $G(m,n)$,
and since $(m-2)(n-2)$ is an odd number, the edges with colors 
$a,b,c$ and $d$ cannot only be incident with vertices of degree four. This implies that for each color $x\in \{a,b,c,d\}$, there exists an edge $e_x$ with color $x$ joining vertices with degrees $4$ and $3$. Thus, all colors $a,b,c$ and $d$ appear in palettes of vertices of degree $3$, which implies that $\check s(G(m,n))\geq 4$.

Finally, we show that if $m n$ is odd, then $\check s(G(m,n)) = 5$.
Suppose, to the contrary, that $\check s(G(m,n))=4$, and
let $\phi$ be a proper edge coloring of $G(m,n)$ with four distinct palettes. 
Throughout the rest of the proof, denote by $M_i$ the color class $i$
under $\phi$, i.e., the set of edges with color $i$ under $\phi$. 

Let us first prove that the number of $3$-element palettes under $\phi$
is at least two. 
Since there are at most two palettes of size $4$, the set $A$
of colors appearing in palettes of size $4$ satisfies $4\leq |A| \leq 8$.
Moreover, $A$ clearly has a partition $\{A_1, A_2, A_3, A_4\}$ such that
$1\leq |A_i| \leq 2$, and each palette of size $4$ contains exactly
one color from $A_i$, $i=1,2,3,4$. Furthermore, 
since $mn$ is odd, there is an odd 
number of vertices of degree $4$ in $G(m,n)$. Therefore,
for every $i \in \{1,2,3,4\}$, there is a color $a_i \in A_i$
and an edge colored $a_i$ that joins vertices of degree $3$ and $4$.
We thus conclude that each of the colors $a_1, a_2, a_3, a_4$ appears
in a palette of size $3$, and it follows that
 the number of palettes of size $3$
is at least two.

Now, since there are at least two palettes of size $3$, there must be
exactly one palette of size $4$ and one palette of size $2$.
Without loss of generality we assume that for each vertex $v\in V(G(m,n))$ with
degree four, we have $S(v,\phi)=\{1,2,3,4\}$, and for each color 
$x\in \{1,2,3,4\}$, there exists an edge $e_x$ with color $x$ joining vertices with degrees $4$ and $3$. 
Thus, all colors $1,2,3$ and $4$ appear in palettes
of vertices of degree three.

Since two distinct palettes occur at vertices of degree three, at most six
colors $1,\dots,6$ are used in the coloring $\phi$.
Suppose first that disjoint palettes occurs at vertices of degree three.
If three colors from $\{1,2,3,4\}$ appear in one such palette,
i.e., if for each vertex $v\in V(G(m,n))$ with degree three, either, say,
$S(v,\phi)=\{1,2,3\}$ or $S(v,\phi)=\{4,5,6\}$, then 
since both $m$ and $n$ are odd, vertices of degree two only have one
possible palette under $\phi$, and neither of colors $5$ and $6$
appear at vertices of degree $4$, this implies that
all vertices with degree three have the same palette, which is a contradiction.
If instead two colors from $\{1,2,3,4\}$ appear in both palettes, e.g.
if for each vertex $v\in V(G(m,n))$ with degree three, 
either $S(v,\phi)=\{1,2,5\}$ or $S(v,\phi)=\{3,4,6\}$, then, again, 
this implies that all vertices with degree three have the same palette, which is a contradiction.

Suppose now instead that
the two distinct palettes at vertices of degree three contain
exactly one common color.
We first consider the case when this common color is in $\{1,2,3,4\}$.
Assume, without loss of generality, that
this color is $3$, and
consider the color class $M_3$.
The edges in $M_3$
either cover all vertices of the graph or all vertices except those with degree two; but this is impossible, since $mn$ and $mn-4$ are both odd numbers.

Suppose now instead that the common color of the different palettes
of vertices of degree three is not in $\{1,2,3,4\}$.
We assume that this common color is $5$, and since all colors
in $\{1,2,3,4\}$ appear on edges incident with vertices of degree $3$,
we may assume, that for each vertex $v\in V(G(m,n))$ with degree three, 
either $S(v,\phi)=\{1,2,5\}$ or $S(v,\phi)=\{3,4,5\}$. 
This means that the color class $M_5$ covers all vertices of the cycle 
$C$ of $G(m,n)$ containing all vertices of degree $3$ and $2$ in $G(m,n)$,
because any path in $G(m,n)$ 
between vertices of degree $2$, whose intermediate vertices
all have degree $3$, has even length.
Now, since all vertices with degree two 
have the same palette, we may assume that for each vertex $v\in V(G(m,n))$ with degree two, $S(v,\phi)=\{a,b\}$.
Since color $5$ appears at each vertex of $C$, we obtain that $a=5$. Without loss of generality we may assume that $b=1$. Let us now consider the color class $M_3$. Clearly, $$|M_3|=\frac{1}{2}\left((m-2)(n-2)+l\right),$$ 
where $l$ is the number of vertices of $C$ with the palette $\{3,4,5\}$. 
Since $(m-2)(n-2)$ is odd, we get that $l$ is odd too. Let $r_3$ and $r_4$ be the number of edges of $C$ with colors $3$ and $4$, respectively. Now we can count the number of vertices of $C$ with the palette $\{3,4,5\}$ using $r_3$ and $r_4$. 
Since, all vertices of degree two have the palette $\{1,5\}$, and color
$5$ does not appear on any edge incident with a vertex of degree four,
$l=2r_3+2r_4$; but this contradicts the fact that $l$ is odd. 

Finally, let us consider the case when the two distinct palettes at vertices
of degree three contain two common colors. Suppose without loss of generality
that
for each vertex $v\in V(G(m,n))$ with degree three, 
either $S(v,\phi)=\{1,2,3\}$ or $S(v,\phi)=\{2,3,4\}$. 
Let us consider vertices with degree two in 
$G(m,n)$;
all such vertices $v$ have the same palette  
$S(v,\phi)=\{a,b\}$.
If $\{a,b\}\cap \{2,3\}\neq \emptyset$, then the color class $M_a$ (or $M_b$) is a perfect matching of $G(m,n)$, which is a contradiction. So, we may assume that $\{a,b\}=\{1,4\}$. 
Let us consider the color class $M_2$. 
Clearly, $$|M_2|=\frac{1}{2}\left((m-2)(n-2)+k+l\right),$$ where $k$ is the number of vertices of $C$ with the palette $\{1,2,3\}$, 
and $l$ is the number of vertices of $C$ with the palette $\{2,3,4\}$. 
Since $(m-2)(n-2)$ is odd, we get that $k+l$ is odd too. On the other hand, 
it is easy to see that $k+l=2(m-2+n-2)$, which is a contradiction.
\end{proof}\


\section{Biregular graphs}

In this section we consider $(a,b)$-biregular graphs. Our primary
aim here is to show that Conjecture \ref{conj:palettebiregular}
holds for several families of biregular graphs.

K\"onig's edge coloring theorem implies that
$\check s(G) \leq 1+ \binom{b}{a}$ for every $(a,b)$-biregular graph $G$
where $a \leq b$.
In particular, this implies that if $G$ is $(b-1,b)$-biregular or
$(1,b)$-biregular, then
$\check s(G) \leq 1+b$, which means that Conjecture
\ref{conj:palettebiregular} holds for all such graphs. In fact,
the latter family of graphs show that the upper bound in
Conjecture \ref{conj:palettebiregular} is in general sharp.

The next lemma will be used frequently.

\begin{lemma}
\label{lemma:lowbound}
    If $G$ is an $(a,b)$-biregular graph with $a<b$, then 
$\check s(G) \geq 1 + \lceil \frac{b}{a} \rceil$.
\end{lemma}
\begin{proof}
	Let $G$ be an $(a,b)$-biregular ($a<b$) graph with bipartition $(X,Y)$
	so that $a|X| = b|Y|$. Consider an arbitrary proper edge coloring of $G$.
	Since any palette of size $a$ appears on at most $|Y|$ vertice, 
	the number of palettes of size $a$ is bounded from below by 
	$\left\lceil\frac{|X|}{|Y|}\right\rceil = 
	\left\lceil\frac{b}{a}\right\rceil$. This implies that 
	$\check s(G) \geq 1 + \lceil \frac{b}{a} \rceil$. 
\end{proof}

The smallest $(a,b)$-biregular graph is
the complete bipartite graph $K_{a,b}$; 
the lower bound in the preceding lemma was obtained in \cite{HornakHudak}
for the case of complete bipartite graphs.
Furthermore,
for complete bipartite
graphs, we have the following; the upper bound shows
that Conjecture \ref{conj:palettebiregular} holds for complete bipartite graphs.
If $a$ and $b$ are positive integers ($a\leq b$),
then we denote the interval of integers from $a$ to $b$
by $[a,b] = \{a, a+1,\dots, b\}$.

\begin{theorem}
\label{thmKmn} If $a<b$ ($a,b\in \mathbb{N}$), then
\begin{center}
$1+\left\lceil \frac{b}{a}\right\rceil \leq
\check{s}\left(K_{a,b}\right)\leq 1+\frac{b}{\gcd(a,b)}$.
\end{center}
\end{theorem}
\begin{proof} 
The lower bound follows from Lemma \ref{lemma:lowbound}.

We set $d=\gcd(a,b)$ and now 
show that $\check{s}\left(K_{a,b}\right)\leq 1+\frac{b}{d}$.
Let $$V\left(K_{a,b}\right)=\{u_{1},\ldots,u_{a},v_{1},\ldots,v_{b}\} 
\text{ and }
E\left(K_{a,b}\right)=\{u_{i}v_{j}: 1\leq i\leq a, 1\leq
j\leq b\}.$$ Also, let $G$ be a subgraph of $K_{a,b}$ induced by 
vertices $\{u_{1},\ldots,u_{d},v_{1},\ldots,v_{d}\}$; 
so $G$ is isomorphic to the graph $K_{d,d}$.

We define an edge coloring $\alpha$ of $G$ as follows: for
$1\leq i\leq d$ and $1\leq j\leq d$, let
\begin{center}
$\alpha\left(u_{i}v_{j}\right)=\left\{
\begin{tabular}{ll}
$i+j-1\pmod{d}$, & if $i+j\neq d+1$,\\
$d$, & if $i+j=d+1$.
\end{tabular}%
\right.$
\end{center}
The coloring $\alpha$ is a proper edge coloring of $G$ and $S_{G}(u_{i},\alpha)=S_{G}(v_{i},\alpha)=[1,d]$ for $1\leq i\leq d$.

Next we construct a proper $b$-edge coloring of
$K_{a,b}$. Before we give the explicit definition of the
coloring, we need two auxiliary functions $f$ and $h$. For $i\in \mathbb{N}$, we
define
$f(i)=1+(i-1)\pmod{d}$ and
for $i,j\in \mathbb{N}$, we define 
\begin{center}
$h(i,j)=\left(\left\lfloor\frac{i-1}{d}\right\rfloor+\left\lfloor\frac{j-1}{d}\right\rfloor\right)\pmod{\frac{b}{d}}$.
\end{center}
Now we define an edge coloring $\beta$ of $K_{a,b}$ by, for
$1\leq i\leq a$ and $1\leq j\leq b$, setting
\begin{center}
$\beta(u_{i}v_{j})=\alpha\left(u_{f(i)}v_{f(j)}\right)+d h(i,j)$.
\end{center}

Let us verify that $\beta$ is a proper $b$-edge coloring of $K_{a,b}$
with exactly $1+\frac{b}{a}$ palettes.
By the definition of $\beta$ and taking into account that
$S_{G}(u_{i},\alpha)=S_{G}(v_{i},\alpha)=[1,d]$ for $1\leq i\leq d$, we have

$$S(u_{i},\beta)=[1,b] \text{ for } 1\leq i\leq a,$$
and
$$S\left(v_{(j-1) d+1},\beta\right)=S\left(v_{(j-1) d+2},\beta\right)=\cdots = S\left(v_{j d},\beta\right) 
= \bigcup_{i=0}^{\frac{a}{d}-1} \{a_id+1,\dots, a_id+d\},$$
for $1\leq j\leq \frac{b}{d}$, and where 
$a_i =\left(i+ \left\lceil\frac{j-1}{d}\right\rceil\right) \pmod{\frac{b}{d}}$.
This implies that $\beta$ is a proper $b$-edge coloring of $K_{a,b}$ with $1+\frac{b}{d}$ distinct palettes.
\end{proof}

From the preceding theorem, we deduce the following, which was
first obtained in \cite{HornakHudak}.

\begin{corollary}
\label{cor:Kab}
If $\gcd(a,b)=a$ ($a<b$), then
$\check{s}\left(K_{a,b}\right)=1+\frac{b}{a}$.
\end{corollary}

\bigskip

In \cite{HornakHudak} the
palette index of the complete bipartite graphs $K_{2,2r}$ was determined;
the following generalization follows from Theorem \ref{prop:bipevendeg}.
Here, and in the following , we assume $r$ to be a positive integer.

\begin{corollary}
\label{cor:2bireg}
	If $G$ is a $(2,2r)$-biregular graph, then $\check s(G) = r+1$.
\end{corollary}
\begin{proof}
	The upper bound follows from Theorem \ref{prop:bipevendeg}.
	The lower bound follows from the fact that assuming that
	at most $r-1$ palettes occur at vertices of degree $2$ implies
	that $G$ has a proper edge coloring with $2r-2$ colors.
\end{proof}

Similarly, we have the following:

\begin{corollary}
\label{cor:2bireg2}
	If $G$ is a $(2r-2,2r)$-biregular graph, then $\check s(G) \leq r+1$.
\end{corollary}

This upper bound is sharp e.g. for complete bipartite graphs of small
order, since $\check s(K_{2,4})=3$ and $\check s(K_{4,6})=4$.

We remark that the two previous corollaries do not only hold for biregular graphs,
but for any bipartite graph where the vertex degrees lie in the set
$\{2, 2r\}$ and $\{2r-2, 2r\}$, respectively.

\bigskip

Our next result on biregular graphs is an easy consequence of a result
on interval colorings. In \cite{HansonLotenToft,KamMir}, it was proved that
every $(2,2r+1)$-biregular graph has an interval coloring using 
$2r+2$ colors.

\begin{proposition}
\label{prop:22r+1}
	If $G$ is a $(2,2r+1)$-biregular graph, then $r+2 \leq \check s(G) \leq 2r+2$.
\end{proposition}

\begin{proof}
	Let $f$ be an interval coloring of $G$ using exactly $2r+2$ colors.
	By taking all colors modulo $2r+1$, we obtain a cyclic interval
	$(2r+1)$-coloring of $G$; such a coloring
	yields at most $2r+2$ distinct palettes in $G$.
	
	The lower bound can be proved as in the proof 
	of Corollary \ref{cor:2bireg}.
\end{proof}

We note that the upper bound in the preceding proposition is sharp, since
$\check s(K_{2,3}) = 4$; in fact it is not hard to see that the upper bound
in Proposition \ref{prop:22r+1} is sharp for all
$(2,3)$-biregular graphs.

\bigskip

Next, we shall establish that Conjecture \ref{conj:palettebiregular}
holds for some families of
biregular graphs with small vertex degrees.
In fact, we shall deduce these results from more general propositions.

Corollary \ref{prop:bipodddegree}
implies that Conjecture \ref{conj:Avesani} holds for
all $(3,3r)$-biregular and $(3r-3, 3r)$-biregular graphs; the upper bound
from Corollary \ref{prop:bipodddegree}
can be slightly improved as follows.

\begin{proposition}
\label{prop:bireg(3,3r)}
	Let $G$ be a bipartite graph.
\begin{itemize}
\item[(i)]	If $G$ is $(3,3r)$-biregular ($r\geq 2$), 
	then $r+1\leq\check s(G) \leq r^{2}+1$.
	
\item [(ii)]	If $G$ is $(3r-3,3r)$-biregular graph ($r\geq 2$), 
	then $\check s(G) \leq r^{2}+1$.
\end{itemize}
\end{proposition}
\begin{proof}
Let us first note that the the lower bound in (i) follows from
Lemma \ref{lemma:lowbound}.

We shall prove the upper bound in (i); the proof of the upper bound in (ii)
is similar.
Consequently, let $G$ be a $(3,3r)$-biregular bipartite graph with 
bipartition $(X,Y)$, 
and let us show that $\check s(G) \leq r^{2}+1$.

Define a new graph $H$ from $G$ by replacing each vertex $y\in Y$ by 
$r$ vertices
$y^{(1)},y^{(2)},\ldots,y^{(r)}$ of degree $3$, 
where each $y^{(i)}$ is adjacent to three neighbors of $y$ in $G$, 
and $y^{(i)}$ and $y^{(j)}$ have disjoint neighborhoods if $i \neq j$.
Clearly, $H$ is a cubic bipartite graph, and so by Hall's matching theorem, $H$ contains a perfect matching $M$.

In the graph $G$, $M$ induces a subgraph $F$ in which each vertex $y\in Y$ has degree $r$ and each vertex $x\in X$ has degree $1$. 
Let us consider the graph $G^{\prime}=G-E(F)$. Since $G^{\prime}$ 
is a $(2,2r)$-biregular graph, by proceeding
as in the proof of 
Theorem \ref{prop:bipevendeg} it can be shown that $G^{\prime}$ has a proper 
$2r$-edge coloring $\alpha$ such that for each $y\in Y$,
$S(y,\alpha)=[1,2r]$, and for each $x\in X$,
$S(x,\alpha)=\{2i-1,2i\}$ for some $i$ ($1\leq i\leq r$). Let us now define an edge coloring $\beta$ of $F$ as follows: for each vertex $y\in Y$, we color the edges of $F$ incident with $y$ with colors $2r+1,2r+2,\ldots,3r$.\\

Finally, we define an edge coloring $\gamma$ of $G$ as follows:

\begin{itemize}

	\item[1)] for every $e\in E(G^{\prime})$, let $\gamma(e)=\alpha(e)$;

	\item[2)] for every $e\in E(F)$, let $\gamma(e)=\beta(e)$.

\end{itemize}

Clearly, $\gamma$ is a proper edge coloring of $G$ with colors
$1,2,\ldots,3r$ such that for each $y\in Y$, $S(y,\gamma)=[1,3r]$, and
for each $x\in X$, $S(x,\gamma)=\{2i-1,2i,2r+j\}$ for some $i,j\in [1,r]$. 
This implies that $\check s(G)\leq r^{2}+1$. 
\end{proof}

We remark that the lower bound in part (i) of Proposition 
\ref{prop:bireg(3,3r)} is sharp by Corollary \ref{cor:Kab}.
Hence, this also holds for parts (i) and (ii) of
the following consequence of Proposition 
\ref{prop:bireg(3,3r)}.

\begin{corollary}
\label{cor:3}
Let $G$ be a bipartite graph.
\begin{itemize}
	\item[(i)] If $G$ is $(3,6)$-biregular, 
	then $3\leq\check s(G) \leq 5$.

	\item[(ii)]  If $G$ is $(3,9)$-biregular, 
	then $4\leq\check s(G) \leq 10$.
	
	\item[(iii)]  If $G$ is $(6,9)$-biregular, 
	then $\check s(G) \leq 10$.
	
	\end{itemize}
\end{corollary}

The preceding result shows that Conjecture \ref{conj:palettebiregular} holds
for some biregular graphs with vertex degrees divisible by three.
Let us now turn to biregular graphs with vertex degrees divisible by four.
In Section 3, we deduced that Conjecture \ref{conj:palettebiregular} holds
for $(4,6)$-biregular graphs.
If $G$ is a $(4,4r)$-biregular or $(4r-4,4r)$-biregular 
graph, then Theorem \ref{prop:bipevendeg}
implies that
$\check s(G) \leq 1+ r(2r-1)$.
Our next proposition yields a slightly better bound.

\begin{proposition}
\label{prop:bireg(4,4r)}
Let $G$ be a bipartite graph.
\begin{itemize}
	
	\item[(i)] If $G$ is $(4,4r)$-biregular ($r\geq 2$),
	then $r+1\leq\check s(G) \leq r^{2}+1$.

	\item[(ii)] If $G$ is $(4r-4,4r)$-biregular  ($r\geq 2$),
	then $\check s(G) \leq r^{2}+1$.
	
\end{itemize}
	
\end{proposition}
\begin{proof}
As in the proof of the preceding proposition, the lower bound
in (i) follows from Lemma \ref{lemma:lowbound}. Let us
prove the upper bound in part (i); part (ii) can be proved
similarly.
Consequently,
let $G$ be a $(4,4r)$-biregular bipartite graph with bipartition $(X,Y)$ 
and let us show that $\check s(G) \leq r^{2}+1$.

Without loss of generality, we may assume that $G$ is connected
(otherwise, we color every component of $G$ as below). Since $G$
is bipartite and all vertex degrees in $G$ are even, $G$ has a
closed Eulerian trail $C$ with an even number of edges. We color the
edges of $G$ with colors ``Red'' and
``Blue'' by traversing the edges
of $G$ along the trail $C$; we color an odd-indexed edge in $C$ with
color Red, and an even-indexed
edge in $C$ with color Blue. Let
$E_{R}$ and $E_{B}$ be the sets of all Red and Blue
edges in $G$, respectively; then $E(G)=E_{R}\cup E_{B}$ and
$E_{R}\cap E_{B} = \emptyset$. Define the subgraphs $G_{R}$ and
$G_{B}$ of $G$ as follows:
\begin{center}
$V\left(G_{R}\right)=V\left(G_{B}\right)=V(G)$ and
$E\left(G_{R}\right)=E_{R}$, $E\left(G_{B}\right)=E_{B}$.
\end{center}

Since $G$ is $(4,4r)$-biregular, each of the subgraphs $G_{R}$ and $G_{B}$ of $G$
is a $(2,2r)$-biregular bipartite graph with bipartition $(X,Y)$. 
Hence, by proceeding as in the proof of the preceding proposition,
we deduce that $G_{R}$ has a proper $2r$-edge coloring $\alpha$ such that for each $y\in Y$ 
$S(y,\alpha)=[1,2r]$, and for each $x\in X$ 
$S(x,\alpha)=\{2i-1,2i\}$ for some $i \in [1,r]$. Similarly, 
$G_{B}$ has a proper $2r$-edge coloring $\beta$ such
that for each $y\in Y$ 
$S(y,\beta)=[1,2r]$, and
for each $x\in X$ 
$S(x,\beta)=\{2j-1,2j\}$ for some $j \in [1,r]$. 
We define a new edge coloring
$\beta^{\prime}$ of $G_{B}$ from $\beta$ as follows: for every $e\in E(G_{B})$, let $\beta^{\prime}(e)=\beta(e)+2r$; then $\beta^{\prime}$
is a proper edge coloring of $G_{B}$ with colors $2r+1,2r+2,\ldots,4r$.
Moreover, for each $y\in Y$, 
$S(y,\beta^{\prime})=[2r+1,4r]$, and for each $x\in X$, 
$S(x,\beta^{\prime})=\{2(r+j)-1,2(r+j)\}$
for some $j \in [1,r]$. \\

Finally, we define an edge coloring $\gamma$ of $G$ as follows:

\begin{itemize}

	\item[1)] for every $e\in E(G_{R})$, let $\gamma(e)=\alpha(e)$;

	\item[2)] for every $e\in E(G_{B})$, let $\gamma(e)=\beta^{\prime}(e)$.

\end{itemize}

Clearly, $\gamma$ is a proper edge coloring of $G$ with colors
$1,2,\ldots,4r$ such that for each $y\in Y$, $S(y,\gamma)=[1,4r]$, and
for each $x\in X$, $S(x,\gamma)=\{2i-1,2i,2(r+j)-1,2(r+j)\}$ for some $i$ and $j$ ($i,j\in [1,r]$). This implies that $\check s(G)\leq r^{2}+1$. 
\end{proof}

Once again, we remark that the lower bound in part (i) of Proposition 
\ref{prop:bireg(4,4r)} is sharp by Corollary \ref{cor:Kab}, so
this also holds for parts (i)-(iii) of
the following consequence of Proposition 
\ref{prop:bireg(4,4r)}.

\begin{corollary}
\label{cor:4}
Let $G$ be a bipartite graph.
\begin{itemize}

	\item[(i)] If $G$ is $(4,8)$-biregular, 
	then $3\leq\check s(G) \leq 5$. 
	
	\item[(ii)] If $G$ is $(4,12)$-biregular,
	then $4 \leq\check s(G) \leq 10$.
	
	\item[(iii)] If $G$ is $(4,16)$-biregular,
	then $5 \leq\check s(G) \leq 17$.
	
	\item[(iv)] If $G$ is $(8,12)$-biregular,
	then $\check s(G) \leq 10$.
	
	\item[(v)] If $G$ is $(12, 16)$-biregular,
	then $\check s(G) \leq 17$.
	
	\end{itemize}
\end{corollary}

Our next result establishes an upper bound on the palette index
of $(5,5r)$-biregular graphs.

\begin{proposition}
\label{prop:bireg(5,5r)}
	If $G$ is a $(5,5r)$-biregular ($r\geq 2$) bipartite graph,
	then $r+1\leq\check s(G) \leq r^{3}+1$.
\end{proposition}
\begin{proof}
The lower bound follows
from Lemma \ref{lemma:lowbound}, so let us prove the upper bound.

Let $G$ be a $(5,5r)$-biregular bipartite graph with bipartition $(X,Y)$,
and let us show that $\check s(G) \leq r^{3}+1$.
As in the proof of Proposition \ref{prop:bireg(3,3r)}, we define a new graph $H$ from $G$ by replacing each vertex $y\in Y$ by 
$r$ vertices $y^{(1)},y^{(2)},\ldots,y^{(r)}$ of degree $5$, where 
each $y^{(i)}$ is adjacent to five neighbors of $y$ in $G$, 
and  $y^{(i)}$ and $y^{(j)}$ have disjoint neighborhoods if $i \neq j$.
Clearly, $H$ is a $5$-regular bipartite graph, and by Hall's 
matching theorem, $H$ contains a perfect matching $M$. 

In the graph $G$, $M$ induces a subgraph $F$ in which each vertex $y\in Y$ has degree $r$ and each vertex $x\in X$ has degree $1$. 
Let us consider the graph $G^{\prime}=G-E(F)$. Since $G^{\prime}$ 
is $(4,4r)$-biregular, by proceeding as in the proof of 
Proposition \ref{prop:bireg(4,4r)}, we can construct
a proper $4r$-edge coloring $\alpha$ of
 $G^{\prime}$  such that for each $y\in Y$,
$S(y,\alpha)=[1,4r]$ and for each $x\in X$, 
$S(x,\alpha)=\{2i-1,2i,2(r+j)-1,2(r+j)\}$ for some
$i,j\in [1,r]$. Let us now define an edge-coloring $\beta$ of $F$ as follows: 
for each vertex $y\in Y$, we color the edges of $F$ incident with $y$ with colors $4r+1,4r+2,\ldots,5r$.\\

Finally, we define an edge coloring $\gamma$ of $G$ as follows:

\begin{itemize}

	\item[1)] for every $e\in E(G^{\prime})$, let $\gamma(e)=\alpha(e)$;

	\item[2)] for every $e\in E(F)$, let $\gamma(e)=\beta(e)$.

\end{itemize}

Clearly, $\gamma$ is a proper edge coloring of $G$ with colors
$1,2,\ldots,5r$ such that for each $y\in Y$, $S(y,\gamma)=[1,5r]$, and
for each $x\in X$, $$S(x,\gamma)=\{2i-1,2i,2(r+j)-1,2(r+j),4r+k\}$$ 
for some $i,j,k\in [1,r]$. 
This implies that $\check s(G)\leq r^{3}+1$. 
\end{proof}

We remark that it is possible to prove a similar upper bound for
$(5r-5,5r)$-biregular graphs. From the preceding proposition we
deduce the following.

\begin{corollary}
		If $G$ is a $(5,10)$-biregular graph,
	then $3 \leq\check s(G) \leq 9$.
\end{corollary}

Again, by Corollary \ref{cor:Kab}, the lower bound 
in the preceding corollary (and in Proposition \ref{prop:bireg(5,5r)}) 
is sharp. This also applies to the next proposition
which concerns $(r,2r)$-biregular graphs. 

\begin{proposition}
\label{prop:bireg(r,2r)}
	If $G$ is an $(r,2r)$-biregular ($r\geq 2$) bipartite graph, then $3\leq\check s(G) \leq 2^{\lceil \frac{r}{2} \rceil}+1$.
\end{proposition}
\begin{proof}
As in the proofs of the preceding propositions, the lower bound follows
from Lemma \ref{lemma:lowbound}.
Let $G$ be an $(r,2r)$-biregular bipartite graph with bipartition $(X,Y)$,
and let us show that $\check s(G) \leq 2^{\lceil \frac{r}{2} \rceil}+1$.
We consider two cases.

\bigskip

\noindent
{\bf Case 1.} {\it $r$ is even:}
Let $r=2k$ ($k\in \mathbb{N}$).  
	Since $G$ is $(2k,4k)$-biregular, it has a decomposition into $k$
	$(2,4)$-biregular graphs $G_{1},\ldots,G_{k}$; this follows by splitting
	vertices of degree $2k$ into two vertices of degree $k$, vertices of degree $4k$ into four vertices of degree $k$, and taking perfect matchings in the resulting $k$-regular bipartite graph. As in the proof of Theorem 
	\ref{prop:bipevendeg} it can be shown that each graph $G_{i}$ has a proper 
$4$-edge 
coloring $\alpha_{i}$ such that for each $y\in Y$,
$S\left(y,\alpha_{i}\right)=[4i-3,4i]$, and for each $x\in X$, either
$S\left(x,\alpha_{i}\right)=\{4i-3,4i-2\}$ or $S\left(x,\alpha_{i}\right)=\{4i-1,4i\}$ ($1\leq i\leq k$). Let us now define an edge-coloring $\beta$ of $G$ as follows: for $1\leq i\leq k$ and for every $e\in E(G_{i})$, 
let $\beta(e)=\alpha_{i}(e)$.

Clearly, $\beta$ is a proper edge coloring of $G$ with colors
$1,2,\ldots,4k$ such that for each $y\in Y$, $S(y,\beta)=[1,4k]$, and
for each $x\in X$, $S(x,\beta)$ is one of $2^{k}$ possible palettes.
This implies that $\check s(G)\leq 2^{k}+1$.

\bigskip

\noindent
{\bf Case 2.} {\it $r$ is odd:}
Let $r=2k+1$ ($k\in \mathbb{N}$). 
	Since $G$ is $(2k+1,4k+2)$-biregular, 
	it has a $(1,2)$-biregular subgraph $F$; this follows by 
	splitting vertices of degree $4k+2$ into two vertices of degree $2k+1$,
	and taking a perfect matching in the resulting $(2k+1)$-regular bipartite graph. Let us consider the graph $G^{\prime}=G-E(F)$. Since $G^{\prime}$ is 
	a $(2k,4k)$-biregular graph, it follows from the proof in Case 1 that $G^{\prime}$ has a proper $4k$-edge coloring $\alpha$ such that for each $y\in Y$,
$S(y,\alpha)=[1,4k]$ and for each $x\in X$, 
$S(x,\alpha)$ is one of $2^{k}$ possible palettes. Let us now define an 
edge coloring $\beta$ of $F$ as follows: for each vertex $y\in Y$, we color the edges of $F$ incident with $y$ with colors $4k+1$ and $4k+2$.\\
	
Finally, we define an edge coloring $\gamma$ of $G$ as follows:

\begin{itemize}

	\item[1)] for every $e\in E(G^{\prime})$, let $\gamma(e)=\alpha(e)$;

	\item[2)] for every $e\in E(F)$, let $\gamma(e)=\beta(e)$.

\end{itemize}

Clearly, $\gamma$ is a proper edge coloring of $G$ with colors
$1,2,\ldots,4k+2$ such that for each $y\in Y$, $S(y,\gamma)=[1,4k+2]$, and
for each $x\in X$,
\begin{center}
either $S(x,\gamma)=S(x,\alpha)\cup \{4k+1\}$ or $S(x,\gamma)=S(x,\alpha)\cup \{4k+2\}$;
\end{center}
thus there are at most $2^{k+1}$ possible choices for the palette $S(x,\gamma)$.
This implies that $\check s(G)\leq 2^{k+1}+1$. 
\end{proof}

\begin{corollary}
\label{cor:bireg()}
	Let $G$ be a bipartite graph.
	\begin{itemize}

	\item[(i)] If $G$ is $(6,12)$-biregular,
	then $3 \leq\check s(G) \leq 9$.
	
	\item[(ii)] If $G$ is a $(8,16)$-biregular,
	then $3 \leq\check s(G) \leq 17$.
		\end{itemize}
\end{corollary}

Our final result for biregular graphs shows that a slightly weaker
form of Conjecture \ref{conj:palettebiregular} holds
for $(3,5)$-biregular graphs.

\begin{proposition}
\label{prop:bireg(3,5)}
	If $G$ is a $(3,5)$-biregular bipartite graph, then $5\leq\check s(G) \leq 7$. 
\end{proposition}
\begin{proof}
Let $G$ be a $(3,5)$-biregular bipartite graph with parts $X$
and $Y$. Since $G$ is $(3,5)$-biregular, we have that $|X| = 5k$
and $|Y|=3k$ for some positive integer $k$.

By Lemma \ref{lemma:lowbound}, we obtain that $\check s(G)\geq 3$. Moreover, if $\check s(G) = 3$, then in
a proper edge coloring attaining this value, vertices in $X$ in $G$
must have two distinct palettes. If $\varphi$ is such a coloring, then
the vertices
of degree five all have the same palette under $\varphi$. This implies
that $\varphi$ is a proper $5$-edge coloring, and so there is some color
appearing at all vertices of degree three in $G$. However, this contradicts
that $\varphi$ is a proper $5$-edge coloring. Hence,
$\check s(G)\geq 4$.

Now assume that $\check s(G) = 4$. Using similar counting arguments as before,
it follows that vertices in $X$ must have at least two distinct palettes.
Vertices in $Y$ must also have at least two distinct palettes, because
suppose there is only one palette $\{1,2,3,4,5\}$ of size $5$ and three palettes of size $3$; then, since there are three palettes 
of size $3$ and no color can appear
in all these three palettes, there is 
exactly one color,
say $1$, that appears in exactly one palette of size $3$, say $\{1,2,3\}$; 
the remaining palettes of size three are then $\{2,4,5\}$ and $\{3,4,5\}$.
Now, since $|X| = 5k$ and $|Y|=3k$, we have that the number of vertices
in $X$ with the palette $\{1,2,3\}$ is $3k$. But then colors $4$ and $5$
appear at all $3k$ vertices of $Y$ but only at $2k$ vertices in $X$,
a contradiction.
Hence, the vertices in $Y$ have at least two distinct palettes, and so,
there are exactly two palettes of vertices in $X$ and two palettes of vertices
in $Y$.

Now, if the two distinct palettes
of vertices in $X$ are not disjoint, then at most $5$ colors are used in a proper
edge coloring of $G$ with a minimum number of palettes, which contradicts that
two distinct palettes appear at vertices in $Y$. Thus 
there is a proper edge coloring $\varphi$ with $4$ distinct palettes,
and where the two palettes
of vertices in $X$ are disjoint, say $\{1,2,3\}$ and $\{4,5,6\}$. Now, since
exactly $6$ colors are used in $\varphi$, and since only two distinct palettes
appear at vertices of $Y$, some color appears at all vertices
of $Y$, say color $1$. This implies that the number of vertices in $X$
with the palette $\{1,2,3\}$ is $|Y| =3k$. However, some color in $\{4,5,6\}$
must also appear at all vertices in $Y$, which implies that the 
number of vertices in $X$ with the palette $\{4,5,6\}$ is $3k$, a contradiction
because $|X| = 5k$. Hence $\check s(G)\geq 5$.

 Let us now show that $\check s(G) \leq 7$.
 By Hall's matching theorem, $G$ has a matching $M$ that saturates all the vertices 
of degree $5$. 
The graph $G^{\prime} = G-M$ is a bipartite graph with $\Delta(G^{\prime})=4$. As in the proof of Corollary \ref{cor:bipdeg4}, $G^{\prime}$ has a proper edge coloring $\alpha$ with colors $1,2,3,4$ such that the vertices of degree $2$ 
in $G'$ have $2$ possible palettes and the vertices of degree $3$ in $G'$
have $4$ possible palettes.\\ 

We now define a proper edge coloring $\beta$ of $G$ as follows: 

\begin{itemize}
		
		\item[1)] for every $e\in E(G^{\prime})$, let $\beta(e)=\alpha(e)$;
	
		\item[2)] for every $e\in M$, let $\beta(e)=5$.
	
\end{itemize}
In the coloring $\beta$ the vertices of degree $5$ in $G$
all have the same palette, the vertices of degree $3$ in $G$
that are covered by $M$ have again $2$ possible palettes and the rest of the vertices of degree $3$ in $G$ have $4$ possible palettes. 
This implies that $\check s(G) \leq 7$. 
\end{proof}

We remark that the lower bound in the preceding proposition is sharp
since $\check s(K_{3,5})= 5$, as proved by \cite{HornakHudak}.\\


\section{Graphs with large palette index}

For every graph $G$ we clearly have $\check s(G) \leq |V(G)|$.
In this section we shall characterize the graphs $G$ with largest possible
palette index in the sense that
$G$ satisfies $\check s(G)=|V(G)|$. Throughout this section we only
consider graphs with no multiple edges.

Denote by $\hat K^{j}_3$ the graph obtained from $K_3$ and $K_{1,j}$ by identifying the
central vertex of $K_{1,j}$ with a vertex of $K_3$.
Moreover, we denote by $\hat K^{j+}_3$ the graph obtained from
$K_3$ and $K_{1,j}$ by adding an edge between the central vertex
of $K_{1,j}$ and some vertex of $K_3$.

\begin{theorem}
\label{th:palettelarge}
	If $G$ is a graph with no isolated vertices, then
	$\check s(G) = |V(G)|$ if and only if $G$ is 
	isomorphic to $K_3$, $K_{1,j}$ with $j \geq 2$,
	$\hat K^{j}_3$ with $j \geq 1$, 
	or one of
	$\hat K^{j+}_3$ and $K_3 \cup K_{1,j}$ with $j \geq 3$.
\end{theorem}

\begin{proof}
		Sufficiency is straightforward, so let us prove necessity.
		Let $G$ be a graph with $\check s(G) = |V(G)|$.
		
		By the pigeonhole principle, there are at least two vertices
		in $G$ that have equal degrees; let us first prove that
		any such pair of vertices have vertex degrees $1$ or $2$.
		Suppose that
		$G$ contains two vertices $u$ and $v$ of equal degree greater than $2$.
		It is straightforward to verify that
		there is a partial edge coloring of $G$ such that an edge of $G$
		is colored if and only if 
		it is incident with $u$ or $v$, and such that $u$ and $v$ have 
		the same palettes.
		However, any proper extension of
		such a partial edge coloring 
		of $G$ (not necessarily using a minimum number of colors)
		produces at most
		$|V(G)|-1$ distinct
		palettes. Thus if two vertices in $G$ have equal degree,
		then they both have degree
		$1$ or $2$.

		Let us first assume that there are two vertices $u$ and 
		$v$ of degree $2$ in $G$. Unless
		$u$ and $v$ are contained in a cycle of length $3$, there is
		a similar partial edge coloring as in the preceding paragraph.
		Moreover, if three vertices of $G$ have degree $2$,
		and these vertices are not contained in a component isomorphic to $K_3$,
		then $\check s(G) < |V(G)|$. Hence, either $G$ contains a component
		isomorphic to $K_3$
		or $G$ contains two vertices of degree $2$ that lie on a cycle of length
		three, and no other vertex of $G$ has degree $2$.
		Let $F$ be the component of $G$ containing $u$ and $v$.
		We shall prove that if $F \ncong K_3$, 
		then $F \cong \hat K^{j}_3$ or $F \cong \hat K^{j+}_3$.
		
		Suppose first that $F$ does not contain any vertices of degree $1$.
		Let $w$ be a vertex of maximum degree in $F$ and assume that 
		$\Delta(F) \geq 3$. 
		Now, since $F$ has no more than one vertex of degree $d$ for each
		$d \in \{3,\dots, \Delta(F)-1\}$, the degree of $w$ is at most
		$2+ |\{3,\dots,\Delta(F)-1\}| = \Delta(F)-1$, a contradiction.	
		We conclude that if $F \ncong K_3$, then
		$F$ must contain some vertex of degree $1$.
				
		Assume, consequently, that $F$ contains some vertex
		of degree $1$. If two vertices of degree $1$ in $F$ have distinct neighbors,
		then there is a proper edge coloring of $F$ where these two vertices have
		the same palette, contradicting that  $\check s(G) = |V(G)|$. 
		Hence, all vertices of degree $1$ in $F$
		are adjacent to a fixed vertex $y$ of $F$. 
		Since $F$ contains vertices of degree $1$, and $u$ and $v$ 
		both have degree $2$,
		$\Delta(F) \geq 3$. 
		Moreover, since all vertices of degree greater than $3$ in $F$ have
		distinct degrees, 
		there is a unique vertex $w$ of maximum degree $\Delta(F)$ in $F$.
		Furthermore, it follows from the same argument that
		$w$ is adjacent to some vertex of degree $1$ in $F$.
		Thus, all vertices of degree $1$ in $F$ are adjacent to $w$.
		If all vertex degrees in $F$ are in the set $\{1,2,\Delta(F)\}$,
		then $G \cong \hat K^{j}_3$.
		
		Suppose that there is some
		vertex $x$ of degree $k$, $3 < k < \Delta(F)$ in $F$.
		Without loss of generality, we assume that $x$ has second largest
		degree in $F$.
		Since all vertices of degree $1$ in $F$ are adjacent to 
		$w$, $x$ must be adjacent
		to $u, v, w$ and exactly $k-3 \geq 1$
		vertices of distinct degrees in the set $\{3,\dots,k-1\}$.
		Therefore, $x$ is adjacent to a vertex $y$ of degree $k-1$. However,
		$y$ can be adjacent only to vertices of degrees
		in the set $\{3,\dots, k-2\} \cup \{k, \Delta(F)\}$, so that
		its degree is at most $k-2$, a contradiction. 
		We conclude that there is no vertex of degree greater
		than $3$ in $F$ except for $w$. Moreover, if all vertex degrees
		of $F$ are in the set $\{1,2,3,\Delta(F)\}$,
		where $\Delta(F) > 3$,
		then $F \cong \hat K^{j+}_3$,
		because all vertices of degree $1$ in $F$ are adjacent to $w$.
			
		We conclude that $u$ and $v$ must lie in
		a component $F$ of $G$ that is isomorphic to $K_3$, $\hat K^{j}_3$
		or $\hat K^{j+}_3$.
		
		\bigskip
		
		Suppose that 
		$G$ has more than one component, and let $H$ be a component of
		$G-V(F)$. 
		Now, since the palette index of $G$ is $|V(G)|$,
		$H$ does not contain any vertex of degree $2$.
		Thus any two vertices of equal degree in 
		$H$ have degree $1$.
		Moreover, by the pigeonhole principle at least 
		two vertices of $H$ have
		equal degree, and it is easy to see that if two vertices 
		$x$ and $y$ of degree $1$ in $H$
		are not adjacent to the same vertex, then $\check s(G) < |V(G)|$.
		We conclude that there are at least two vertices of degree
		one in $H$ that are adjacent to
		the same vertex in $H$ (unless $H$ consists of a single edge). 
		Since all other vertex degrees in $H$ are different,
		all vertices of degree $1$ in $H$ are adjacent 
		to the vertex of maximum degree in $H$.
		Moreover, it follows, as in the preceding paragraph, that
		the only vertex degrees in $H$
		are $\Delta(H)$ and $1$; and so, $H$ is isomorphic to a star. 
		
		Now, if there are vertices of degree $1$ in different 
		components of $G$, then clearly
		$\check s(G) < |V(G)|$. 
		Thus, if $G-V(F)$ is non-empty, then $F \cong K_3$ and 
		$G-V(F)$ is a star.
		
		We conclude that $G$ is isomorphic to $K_3$, 
		$\hat K^{j}_3$, $\hat K^{j+}_3$
		or to the disjoint union of $K_3$ and a star.
		
		\bigskip
		
		The case when there are 
		no two vertices of degree $2$ in $G$, can be dealt with similarly
		by first deducing that two vertices in $G$ have degree $1$, and, 
		as before, all such vertices of $G$
		are adjacent to the vertex of maximum degree in $G$. 
		By proceeding as above it is now easy to prove that
		the only vertex degrees in $G$ are $1$ and $\Delta(G)$,
		and that $G$ must be connected.
		Hence, $G$ is isomorphic to a star.
\end{proof}

Although the preceding theorem only holds for graphs with no
isolated vertices, we note that if $G$ is a graph with no isolated vertices
and $\check s(G) = |V(G)|$, then $\check s(G \cup K_1) = |V(G \cup K_1)|$.

Consider a non-regular graph $G$ which is 
the union of two regular edge-disjoint Class 1 graphs
$H_1$ and $H_2$ satisfying that $V(H_1) \subseteq V(H_2)$.
Since both $H_1$ and $H_2$ are Class 1 and $G$ is non-regular, we have that
$\check s(G)=2$.
It is not difficult to see that
the converse holds as well. Indeed, assume that $G$ is a graph with
$\check s(G) =2$, and
let $\phi$ be a proper edge coloring of $G$ attaining this minimum.

It follows from a result of \cite{HornakKalinowskiMeszkaWozniak}
that $G$ is not regular, and
thus exactly two different vertex degrees appear in $G$;
$d_1$ and $d_2$, say, where $d_1 > d_2$. 
For $i=1,2$, let $C_i(\phi)$ be the set of all colors
appearing on edges incident with vertices of degree $d_i$ under $\phi$.
If $C_2(\phi) \nsubseteq C_1(\phi)$, then there is some color $j \in C_2(\phi)$
which does not appear at any vertex of degree $d_1$ in $G$, and
since $|C_1(\phi)| > |C_2(\phi)|$, there is some color $k$
which does not appear on any edge incident with a vertex of degree $d_2$.
Hence, by recoloring all edges with color $j$ by color $k$,
we obtain, from $\phi$, a proper edge coloring $\phi'$ of $G$ with
two distinct palettes, and
such that 
$$|C_2(\phi') \setminus C_1(\phi')| < |C_2(\phi) \setminus C_1(\phi)|.$$
We conclude that we may assume that $C_2(\phi) \subseteq C_1(\phi)$.
Now, let $H_1$ be the edge-induced subgraph of $G$ induced by all
edges with colors in $C_1(\phi) \setminus C_2(\phi)$, and let $H_2$
be the edge-induced subgraph of $G$ induced by all
edges with colors in $C_2(\phi)$. The graph $H_1$ is a regular Class 1 graph
and the graph $H_2$ is a regular Class 1 graph.
Moreover, since $C_2(\phi) \subseteq C_1(\phi)$, $V(H_1) \subseteq V(H_2)$.
We have thus proved the following.

\begin{proposition}
	If $G$ is a graph, then $\check s(G) =2$ if and only if
	$G$ is a non-regular graph which is the
	union of two regular edge-disjoint Class 1 graphs $H_1$ and $H_2$,
	satisfying that $V(H_1) \subseteq V(H_2)$.
\end{proposition}

A partial characterization of graphs with palette index $3$ was obtained in
\cite{BonviciniMazzuoccolo}.
We would like to pose the following question.

\begin{problem}
	Is it possible to characterize graphs
	$G$ satisfying that ${\check s}(G) = |V(G)|-1$?
\end{problem}

\acknowledgements
\label{sec:ack}
The authors would like to thank the referees for helpful comments and suggestions,
particularly for pointing out an argument which simplified the proof of
Theorem 3.6.
The second author would like to thank Hrant Khachatrian for helpful comments and remarks.

\nocite{*}
\bibliographystyle{abbrvnat}
\bibliography{Palette}
\label{sec:biblio}

\end{document}